\documentclass[20pt]{article}
\usepackage{multirow}

\usepackage[utf8]{inputenc}
\usepackage{amsmath}
\usepackage{amssymb}
\usepackage{mathrsfs}
\usepackage{amsthm}
\usepackage{cite}
\usepackage{hyperref}
\usepackage{enumerate}
\usepackage{combelow}
\usepackage[left=0.7in,right=0.7in,top=0.7in,bottom=0.7in]{geometry}
\usepackage{titlesec}
\titleformat*{\section}{\large\bfseries}
\newtheorem{theorem}{Theorem}[section]
\newtheorem{lemma}[theorem]{Lemma}

\newtheorem{definition}[theorem]{Definition}
\newtheorem{example}[theorem]{Example}

\newtheorem{remark}[theorem]{Remark}
\numberwithin{equation}{section}
\title{Fixed Point Theorems Using Interpolative Boyd-Wong Type Contractions  And Interpolative Matkowski Type Contractions on Partial $S_{b}$-Metric Space  }
\author{\large Anuradha Gupta$^1$ and Rahul Mansotra$^2$\\ $^1$ {\small Department of Mathematics, Delhi College of Arts and Commerce,}\\ {\small University of Delhi, Netaji Nagar, New Delhi-110023, India.}\\{\small E-mail:dishna2@yahoo.in}\\{\small $^2$Department of Mathematics, University of Delhi, New Delhi-110007, India.}\\{\small E-mail:mansotrarahul2@gmail.com}}
\date{}
\begin{document}
\maketitle
\begin{abstract}
In this article, we define and explore the topological properties of partial $S_{b}$-metric space. We define  interpolative Boyd-Wong type contraction and  interpolative Matkowski
type contractions in the setting of partial  $S_{b}$-metric space and obtain fixed point results for the same. 
\end{abstract}\section{Introduction and Preliminaries} A point $x \in W$ is said to be a fixed point of a mapping $f$ on $W$ if $f(x)=x.$  The Banach contraction principle\cite{bm2} is the 
 first  significant result in the fixed point theory. Fixed point theory has  immense applications in wide range of fields which attracts the researchers around the world.  Nowadays, metric spaces has been  generalized by researchers by relaxing condition in its definitions. Some of them are  partial $b$-metric space\cite{bm12}, $S-$metric space\cite{bm16}, $S_{b}$-metric space\cite{bm15} , partial $S_{b}$-metric space\cite{bm13}   etc. Topology plays an important role in discussing the structure of the space by studying various properties like connectedness, convergence, continuity, compactness etc. in abstract spaces. Many authors studied the corresponding topology generated by the generalization of metric spaces. For example, Dhage \cite{bm4} explored the topological charateristics of  D-metric space. Mustafa et al. \cite{bm17} defined topology on  partial $b-$metric space. In this section, we study the topological structure of
 partial  $S_{b}$-metric space. 
\par Recently, S. Sedghi et.al.\cite{bm14} gave the definition of $S-$metric space.
\begin{definition}\cite{bm14}
    Let $W$ be a nonempty set. A map  $S$ from $W\times W \times W$ to $[0,\infty)$ defines  a $S-$metric on $W$ if for all $ u, v,w,p \in W,$ Satisfies the following properties:
\begin{enumerate}
        \item $S(u,v,w)= 0$ iff $u=v=w,$
        \item $S(u,v,w) \leq S(u,u,p) + S(v,v,p) +S(w,w,p)$.
    \end{enumerate}
    \end{definition}
    Earlier, N. Nabil\cite{bm11} gave the notion of partial $S$-metric space by combining the concept of   $S-$metric space and partial metric space. \begin{definition}\cite{bm11}
        Let W be a nonempty set. A map $S$ from $W\times W \times W$ to $[0,\infty)$ defines a partial $S$-metric on W if for all $ u, v,w,r \in W,$ Satisfies the following properties; 
        \begin{enumerate}
            \item $u=v$ iff $S(u,v,w)=S(u,u,u)=S(v,v,v)=S(w,w,w),$
            \item $S(u,u,u) \leq S(u,v,w),$
            \item $S(u,u,v)=S(v,v,u),$
            \item $S(u,v,w) \leq S(u,u,r)+S(v,v,r)+S(w,w,r) -S(r,r,r)$
        \end{enumerate}
    \end{definition}
    N. Souayah, and N. Mlaiki\cite{bm15} introduced $S_{b}$-metric space with the help of $b-$metric space and $S-$metric space.\begin{definition}\cite{bm15}
        Let W be a non empty set. A map $S$ from $W\times W \times W$ to $[0,\infty)$ defines a  $S_{b}$-metric  with coefficient $  t \geq 1$ on $W$ if for all $ u, v,w,r \in W,$  Satisfies the following properties; 
        \begin{enumerate}
            \item $w=u=v$ iff $S(u,v,w)=0$,
            \item $S(u,u,v) = S(v,v,u),$
            \item $S(u,v,w) \leq t(S(u,u,r)+S(v,v,r)+S(w,w,r)).$
        \end{enumerate}
    \end{definition}
    Recently,  N. Souayah \cite{bm13} gave the notion of partial $S_{b}$-metric space by utilizing the idea of   $b-$metric space and partial metric space.   \begin{definition}\cite{bm13}
         Let $W$ be a non empty set. A map $\wp$ from $W\times W \times W$ to $[0,\infty)$ defines a  partial $S_{b}$-metric  with coefficient $  t \geq 1$ on $W$ if  for all $  p,q,r,s \in W,$ Satisfies the following properties; 
        \begin{enumerate}
            \item $p=q=r$ iff $\wp(p,q,r)=\wp(p,p,p)=\wp(q,q,q)=\wp(r,r,r)$,
            \item $\wp(p,p,p)\leq \wp(p,q,r),$
            \item $\wp(p,p,q) = \wp(q,q,p),$
            \item $\wp(p,q,r) \leq t(\wp(p,p,s)+\wp(q,q,s)+\wp(r,r,s))-\wp(s,s,s).$\end{enumerate}
    \end{definition} \begin{definition}\cite{bm13}
        Let $(W,\wp)$ be a partial  $S_{b}$-metric space. Then;
        \begin{enumerate}
            \item A sequence $\{ w_{k}\}$ in $W$  is said to be converges to $w$ if $\lim_{k} \wp(w_{k},w_{k},w) = \wp(w,w,w).$
            \item A sequence  $\{ w_{k}\}$ in $W$  is said to be Cauchy in $W$ if $\lim_{k,l \rightarrow \infty} \wp(w_{k},w_{k},w_{l})$ exists.
            \item A partial  $S_{b}$-metric space 
            $(W,\wp)$ is called complete if for each Cauchy sequence $\{w_{k}\}$ in $W$,  there is $w$  in $W$ such that $\lim_{k,l \rightarrow \infty} \wp(w_{k},w_{k},w_{l})=\lim_{k\rightarrow \infty} \wp(w_{k},w_{k},w)= \wp(w,w,w).$ 
        \end{enumerate}
    \end{definition}
\par We know that Banach  contraction mapping is continous and has a unique fixed point. Does a discontiuous map with  contractive condition always have a fixed point ? This query was answered by Kannan \cite{bm5,bm6}. Karapinar in 
2018 introduced interpolative 
 Kannan type contraction\cite{bm8}, which gives  a new research direction in the  fixed point theory by interpolative approach. A survey of interpolative  contractions can be found in \cite{bm5} .
 \par In 2020, Karapinar\cite{bm1}  gave the definition  of   interpolative Matkowski type contractions and interpolative Boyd-Wong type contractions in metric space. He also gave the fixed point theorems on them.\\
Following notations are given in \cite{bm3}.\\
 $ \Re =\{ \tau | \tau : [0,\infty) \rightarrow [0,\infty),\tau(0) = 0, \tau (t) < t, \mbox{ for } t > 0, \tau \mbox{ is upper semi-continuous} \},$\\
  $\mathbb{N}$ denotes the set of natural numbers, \\
$\mathcal{G}=\{\mathcal{J}|\mathcal{J} [0,\infty) \rightarrow [0,\infty), \mathcal{J} \mbox{ is monotonically increasing }, \lim_{k\rightarrow \infty}\mathcal{J}^{k}(v)=0, \mbox{ for } v > 0\}.$\\
$Fix(S)=\{x|S:W\rightarrow W \mbox{ such that } S(x)=x\}$.\\
Matkowski\cite{bm9,bm10} gave a simple but valuable result as follows:
\begin{lemma}\cite{bm9}
    Let $\mathcal{J} \in \mathcal{G}.$ Then $\mathcal{J}(v) < v $ for each $v > 0$ and $\mathcal{J}(0) =0.$
\end{lemma}
\begin{definition}\cite{bm1} Let $(W, \wp )$ be a metric space. Then the operator $S$ from $W$ to $W$ is called an interpolative Boyd-Wong type contraction in   metric space, if there are $p,q,r\in (0,1)$ such that $p+q+r < 1$ and a monotonically increasing function $\tau \in \Re$ with \\
  $\wp(W(u), W(v)) \leq \tau  ([\wp(u,v)]^{p}[\wp(u,W(u)]^{q}[\wp(v, W(v))]^{r}[\frac{1}{2t} (\wp(W(u),v)+\wp(u,W(v))]^{1-p-q-r}) \mbox{ for each } u,v \in W \backslash Fix(S).$
\end{definition}\begin{theorem}\cite{bm1}
  Let $(W, \wp )$ be a  complete metric space and   $S$ be an interpolative Boyd-Wong type contraction. Then $S$ has a  fixed point in $W$. 
\end{theorem}
\begin{definition}\cite{bm1} Let $(W, \wp )$ be a metric space. Then the operator $S$ from $W$ to $W$ is called an interpolative Matkowski type contraction in   metric space, if there are $p,q,r \in (0,1)$ such that $p+q+r < 1$ and a monotonically increasing function $\mathcal{J} \in \mathcal{G}$ with \\
  $\wp(W(u), W(v))  \leq \mathcal{J} ([\wp(u,v)]^{p}[\wp(u,W(u)]^{q}[\wp(v, W(v))]^{r}[\frac{1}{2t} (\wp(W(u), v)+\wp(u,W(v))]^{1-p-q-r})
 \mbox{ for each } u,v \in W \backslash Fix(S).$
\end{definition}
\begin{theorem}\cite{bm1}
  Let $(W, \wp )$ be a  complete metric space and   $S$ be an interpolative Matkowski type contraction. Then $S$ has a  fixed point in $W$. 
\end{theorem}
This paper comprises of two sections followed by illustrations. The first section presents the topological properties of partial $S_{b}-$metric space while the second section presents the fixed point results  using  interpolative Boyd-Wong type contractions  and interpolative Matkowski type contractions in partial $S_{b}-$metric space.
\par Souayah\cite{bm13} defines the open balls for topological structure of partial $S_{b}$-metric 
 space . Using this concept, we give  the topology on partial $S_{b}-$metric space and subsequently queried their topological properties. Further, with  the  help of interpolative Boyd-Wong type contractions  and interpolative Matkowski type contractions in metric space given by Aydi and Karapinar\cite{bm1}, we give this concept on  partial $S_{b}-$metric space and prove the corresponding fixed point theorems on them.
\section{Partial $S_{b}$-metric Space as Topological Space}
     The various  properties of functions like  continuity, convergence and many more depends on the  corresponding topology of a metric space. Dhage \cite{bm4} examined the topological features of D-metric space. Mustafa et al. \cite{bm17} gave a 
topology on partial $b-$metric space. Shaban Sedghi and Nguyen Van Dung\cite{bm16} gave the topological structure on $S-$metric space. Motivated by the works of Dhage et.al.\cite{bm4, bm17, bm16}, we establish the topological structure of partial  $S_{b}-$metric space. 
\begin{definition} \cite{bm13}
     Let $(W,\wp)$ be a partial  $S_{b}-$metric space. A ball in partial  $S_{b}-$metric space is \\
$D(p;r) = \{ z \in W :~ \wp(p,p,z) < r + \wp(p,p,p) \}$.\end{definition}
\begin{example}
     Let $W= [1,\infty)$. Then we define $\wp: W \times W \times W\rightarrow [0,\infty)$ as follows: \\
$\wp(p,q,r)=\begin{cases}
    p^{5}, & \text{if $p=q=r$}.\\
2(p^{5}+r^{5}), & \text{if $p=q\neq r$}.\\
p^{5}+q^{5}+r^{5}, & \text{otherwise}.
  \end{cases} $\\ 
  Clearly, $(W, \wp)$ is a  partial  $S_{b}$-metric space with coefficient $t=1 $. Note that $D(1;3)=\{1\}.$
  \end{example} 
\begin{example}
     Let $W = \{ 1,2\}$. Then we define $\wp: W\times W \times W \rightarrow [0,\infty) $ as follows:
     \\ 
     $\wp(1,1,1) =\wp(1,1,2)= \wp(2,2,1)=8,~ \wp(1,2,1)=\wp(2,1,1)=\wp(1,2,2)=8,~\wp(2,2,2)=\wp(2,1,2)=4$. Clearly, $W$ is a partial $S_{b}-$metric space with coefficient $t=1$.
     Note that $D(1;r)= W$ for any $r>0$ and $D(2;1)=\{2\}.$
\end{example}
\begin{theorem}
Let $(W, \wp )$ be a partial  $S_{b}$-metric space having coefficient $r \geq 1$, then for
any $x \in W$ and $s > 0$ if $v \in D (x; s )$, then there exists $c > 0$ such that $D(v; c ) \subset D(x; s)$.
\end{theorem}
\begin{proof}
    Suppose that $v \in D (x; s )$. Consider the two cases:\\
    Case $1$: If $v = x$ then take $c = s$. \\
    Case $2$: If $v \neq x$.
   then let $c = \frac{s}{4r}$. Consider the set \\
  $ B = \{ m\in \mathbb{N}: 2r \wp(x,x,v)+ (r-1)\wp(v,v,v) - \wp(x,x,x) > \frac{s}{2^{m+2}} \}$\\
 Thus, by Archimedian property, $B$ is non-empty. Now, by well ordering principle, $B$ has the smallest element $\lq i$' (say) such that $i-1 \notin B$, which implies that  \begin{equation}2r \wp(x,x,v)+ (r-1)\wp(v,v,v) - \wp(x,x,x) \leq \frac{s}{2^{i+1}}.\end{equation}
 Let $u \in D(v; c )$. We need to show that $u \in D(x; s)$ i.e.
 $\wp(x,x,u) < s + \wp(x,x,x).$
\begin{align}&\mbox{Since } u \in D(v; c ), \mbox{ therefore, } \wp(v,v,u) < c + \wp(v,v,v).\\
&\mbox{Also, } v \in D(x; s ), \mbox{ therefore, } \wp(x,x,v) < s + \wp(x,x,x).\hspace{8.5cm}
\end{align}
Now, $\wp(x,x,u) \leq r (\wp(x,x,v)+\wp(x,x,v)+\wp(u,u,v))-\wp(v,v,v)$\\
$~\hspace{2.5cm}  < r (\wp(x,x,v)+\wp(x,x,v)+ c + \wp(v,v,v))-\wp(v,v,v)~\hspace{2cm}$ using inequality $(2.2)$\\
$~\hspace{2.5cm} < 2r (\wp(x,x,v))+   (r-1)(\wp(v,v,v)) + \frac{s}{4} $\\
$~\hspace{2.5cm} <  \wp(x,x,x) +  \frac{s}{2^{i+1}}+\frac{s}{4}~\hspace{6.7cm}$ using inequality $(2.3)$\\
$~\hspace{2.5cm}< \wp(x,x,x) + s.$\\
Thus, we get $\wp(x,x,u) < s + \wp(x,x,x).$ 
Hence,  $u \in D(x; s)$.
Therefore, $D(v; c ) \subset D(x; s)$.
\end{proof}
\begin{theorem}
    The collection $\{ D(x;r) : x \in W , r > 0  \} $ form a basis for some topology $ \tau_{S_{b}}$ (say) on $W$.
\end{theorem}
\begin{proof}Let $\mathscr{B} = \{ D(x;r) : x \in W , r > 0\}$.
It is sufficient  to prove the following two :
\\ 
$(i)$ for each $u \in W$, there exists $B \in \mathscr{B}$ such that $u \in B$.\\
$(ii)$ If $w \in B_{1} \cap B_{2}$ for some $B_{1} , B_{2} \in \mathscr{B}$, then there exists $B_{3}  \in \mathscr{B}$ such that $w \in B_{3}$ and $B_{3} \subset B_{1} \cap B_{2}$.\\
Note that for each $u \in W,$ 
we have $\wp(u,u,u) < r + \wp(u,u,u).$ Thus, $u \in B(u;r).$ \\
Also, if $w \in B_{1}\cap B_{2}$ for some $B_{1} , B_{2} \in \mathscr{B}$, then $w \in B_{1}$ and $w \in B_{2}.$ Thus, by Theorem $2.4$, there exists $r_{1}$ and $r_{2}$ such that $D(w,r_{1}) \subset B_{1}$ and $D(w,r_{2}) \subset B_{2}$.\\
Let $r = \min \{ r_{1}, r_{2}\}.$ Then, $D(w,r)  \subset B_{1} \cap B_{2}.$
\begin{remark}
    $\tau _{S_{b}} = \{ \cup_{\alpha} B_{\alpha} : B_{\alpha} \in \mathscr{B}\}$ is the topology generated by  the collection $\mathscr{B}$ on $W$.
\end{remark}
\par Thus, $(W,    \tau _{S_{b}})$ is called partial  $S_{b}-$ metric topological space and     $\tau _{S_{b}}$ is known as partial  $S_{b}-$ metric topology on W.
\end{proof}
\begin{definition}
    A space $W$ is called $T_{0}$- space if for a pair of  distinct points in $W$, there exists an open set having
one point but excluding the other.
\end{definition}
\begin{definition}
    A space $W$ is called $T_{1}$-space if for each pair  $(u, v)$ in $W$ with $u\neq v$, there exists a pair of open sets $(U,V)$ with $u \in U $ but $ v\notin U$ and $v \in V $ but $ u \notin V.$
\end{definition}
\begin{definition}
    A space W is called Hausdroff (or $T_{2}$) space if for each pair $(w, z)$ in $W$ with $w\neq z$,  there exist
pair of disjoint open sets $(W,Z)$ with $w\in W$ and $z \in Z$.
\end{definition}
 \begin{theorem}
     A partial  $S_{b}$-metric topological space $(W,\wp)$ is  a $T_{0}$ space.
 \end{theorem}
\begin{proof}
     let $u$ and $v$ be two distinct points in $W$. We need to show that there is an open ball $D$ such that $u \in D$ but $v \notin D$ or $v \in D$ but $u \notin D$. We know that  $\wp(u,u,v) \geq \wp(u,u,u)$. Consider the following two cases:\\
     Case $1$: If 
 $\wp(u,u,v) > \wp(u,u,u)$ then  we can choose $r> 0$ such that $\wp(u,u,v) > \wp(u,u,u) + r.$ This implies $v \notin D(u;r)$ but $u \in D(u;r).$ Thus, in that case we get $W$ is a $T_{0}$ space. \\     
 Case $2$: If $\wp(u,u,v)= \wp(u,u,u)$ then $\wp(v,v,v)  \neq \wp(u,u,v)$, otherwise  $\wp(u,u,v)= \wp(u,u,u)=\wp(v,v,v),$ which implies $u =v$, by the definition of partial  $S_{b}$-metric space, which is absurd as $u \neq v$.\\
 Since $\wp(u,u,v)=\wp(v,v,u)$, $\wp(v,v,v) < \wp(v,v,u)$. Thus, we can choose $r> 0$ such that $\wp(v,v,v) +r < \wp(v,v,u)$. Hence, $u \notin D(v;r)$ but $v \in D(v;r)$.
 Thus,  we get $W$ is a $T_{0}$ space.
 Therefore, in both the cases we get $W$ is a $T_{0}$ space.
 Hence, partial  $S_{b}-$metric topological space is a $T_{0}$ space.  \end{proof}
 \par The following example  shows that 
 partial $S_{b}-$metric topological space need not be  a $T_{1}$ or $T_{2}$ space.
 \begin{example}
     Let $W = \{ 1,2\}$. We define $\wp: W\times W \times W \rightarrow [0,\infty) $ as follows :
     \\ 
     $\wp(1,1,1) =\wp(1,1,2)=\wp(2,2,1)=8,~ \wp(1,2,1)=\wp(2,1,1)=\wp(1,2,2)=8,~\wp(2,2,2)=\wp(2,1,2)=4$. Clearly, $W$ is a partial $S_{b}-$metric space.
     Note that  $D(1;r)= W$ for any $r>0$. Thus we cannot find two open sets $U,V$ such that $1 \in U$ but $2\notin U$ and $2 \in V$ but $1\notin V.$
 \end{example}
\begin{definition}
     A topological space W is called compact space if every open cover of W has a finite subcover.
 \end{definition}
\par The next example is given to show that  partial $S_{b}-$metric topological space need not be a compact space.
 \begin{example}
  Let $W=[1,\infty)$. We define $\wp: W \times W \times W\rightarrow [0,\infty)$ as follows: \\
$\wp(p,q,r)=\begin{cases}
    p^{5}, & \text{if $p=q=r$}.\\
2(p^{5}+r^{5}), & \text{if $p=q\neq r$}.\\
p^{5}+q^{5}+r^{5}, & \text{otherwise}.
  \end{cases} $\\ 
  Clearly, $(W, \wp)$ is a  partial  $S_{b}$-metric space with coefficient $t\geq1 $. \\
  Note that $ \mathbb{B} = \{D(1,n): n \geq 3 \} $  is an open cover of $W$ but there is no finite subcover of $\mathbb{B}$ which covers $W.$ Hence W is not compact space.
  \end{example}
\par It may be found that  partial $S_{b}-$metric topological space need not be a connected space. Following example is given to prove that partial $S_{b}-$metric topological space need not be a connected space.
      \begin{example}
     Let $W = \{ 1,2\}$. Define $\wp: W\times W \times W \rightarrow [0,\infty) $ as:
     \\ 
$\wp(1,1,1)=\wp(2,2,2)=4,~\wp(1,1,2)=\wp(2,2,1)=\wp(1,2,1)=\wp(2,1,1)=\wp(1,2,2)=\wp(2,1,2)=8$. Clearly, $W$ is a partial  $S_{b}-$metric space with coefficient $t=1$.\\
    Note $D(1;\frac{1}{2}) = \{1\}$ and $D(2;3) = \{2\}.$ Hence, $W =  D(1;\frac{1}{2}) \cup D(2;3).$ 
Thus, W is not a connected space.
     \end{example}
\section{Fixed Point Theorems Using Interpolative Boyd-Wong Type Contractions  and Interpolative Matkowski Type Contractions in Partial $S_{b}$-metric Space}
  Motivated by the studies of Eral karapinar et.al\cite{bm1} on interpolative Boyd-Wong type contractions  and interpolative Matkowski type contractions in metric space. We have given the concept of interpolative Boyd-Wong type contraction  and interpolative Matkowski type contraction in the setting of partial  $S_{b}$-metric space. Further, we have provided the fixed point results related to this contractions.
\begin{definition}
  Let $(W, \wp )$ be a partial  $S_{b}$-metric space with coefficient $t \geq 1$. Then the operator $S$ from $W$ to $W$ is called an interpolative Boyd-Wong type contraction in    partial  $S_{b}$-metric space, if there are $p,q,r,s \in (0,1)$ such that $p+q+r+s < 1$ and a monotonically increasing function $\tau\in \Re$ with 
  \begin{align}
  \wp(S(a), S(b), S(c))  \leq \tau ([\wp(a,b,c)]^{p}[\wp(a,a,S(a))]^{q}[\wp(b,b,& S(b))]^{r}[\wp(c,c,S(c))]^{s}\notag\\&[\frac{1}{2t} (\wp(S(a),S(a), b)+\wp(S(b),S(b), c)]^{1-p-q-r-s})
  \end{align}
  for all $a,b,c \in W \backslash Fix(S).$
\end{definition}
\begin{example}
Let $W= \{0,3\} \cup [4,\infty)$. Define $\wp: W \times W \times W \rightarrow [0,\infty)$ as follows: 
\begin{center}$ \wp(a,b,c)=\begin{cases}
    a^{5}, & \text{if $a=b=c$},\\
2(a^{5}+c^{5}), & \text{if $a=b\neq c$},\\
a^{5}+b^{5}+c^{5}, & \text{otherwise}.
  \end{cases} $\end{center} 
  Clearly, $(W, \wp)$ is a  partial  $S_{b}$-metric space with coefficient $t=1$. Define $\tau : [0,\infty) \rightarrow [0,\infty)$ as  follows:  \begin{center}
    $\tau(a)=\begin{cases}
    \frac{9a}{10}, & \text{if $a\in [0,1]$},\\
    \frac{a}{2}, & \text{if $a>1$}.
  \end{cases}$\end{center} Clearly, $\tau  \in \Re$.
  Also, we define $S : W \rightarrow W$ as  follows:\begin{center} $S (a)=\begin{cases}
    0, & \text{if $a\in \{0,3\}$},\\
    3, & \text{otherwise}.
  \end{cases}$\end{center}
   Claim. $S$ is an interpolative Boyd-Wong type contraction with $p=q=r=s=\frac{1}{5}$.\\
   For this, we have to consider the following cases:\\
   \begin{tabular}{ |p{6cm}|p{4.2cm}|p{6.1cm}| }
 \hline
 \multicolumn{3}{|c|}{Case $1$: $a=b=c$ } \\
 \hline
 & Value of $\wp(S(a), S(b), S(c))$ & Lower bound of R.H.S of inequality $(3.1)$\\
 \hline
 Subcase $1(i)$: $a=b=c=3$  & 0   &0\\
 \hline 
 Subcase $1(ii)$: $a=b=c\neq 3$ &   $243$  &    $607.08 $\\
 \hline 
\multicolumn{3}{|c|}{ Case $2$. If $a=b\neq c$} \\
 \hline 
  Subcase $2(i)$: $a=b=3$ and $c\neq 3$  & $486$ & $569.773$\\
  \hline
 Subcase $2(ii)$:  $a=b\neq 3$ and  $c = 3$  &  $486$ & $807.40$  \\
\hline
Subcase $2(iii)$: $a=b\neq 3$ and $ c \neq 3$ & $243$  & $1214.17$ \\
 \hline 
\multicolumn{3}{|c|}{Case $3$: $a\neq b$ and $a=c$ } \\
\hline
 Subcase $3(i)$: $b\neq 3 $ and $a=c=3$ & $243$ & $499.97$\\
 \hline  Subcase $3(ii)$: $b= 3 $ and $a=c\neq3$ & $486$ & $741.19$\\
 \hline Subcase $3(iii)$: $b\neq  3 $ and $a=c\neq3$ & $243$ & $1294.82$\\ 
  \hline 
\multicolumn{3}{|c|}{Case $4$: $a\neq b$ and $b=c$ }\\
\hline Subcase $4(i)$: $b=c= 3 $ and $a\neq3$ & $243$ & $399.13$\\
\hline Subcase $4(ii)$: $b=c\neq 3 $ and $a=3$ & $486$ & $787.84$ \\
\hline Subcase $4(iii)$: $b=c\neq 3$ and $a\neq3$ & $243$ & $1315.96$\\
  \hline 
\multicolumn{3}{|c|}{ Case $5$: $a\neq b \neq c$ }\\
\hline Subcase $5(i)$: a=3 & $486$ & $872.72$\\ 
\hline  Subcase $5(ii)$: b=3 & $486$ & $758.18$\\
\hline Subcase $5(iii)$: $c=3$  & $486$ & $1051.22$ \\ 
\hline Subcase $5(iv)$: $a\neq b \neq c \neq 3$ & $243$& $1315.96$\\
\hline
\end{tabular}
 \\ Thus, all the cases satisfies inequality $(3.1)$. \\
Therefore, $S$ is an  interpolative Boyd-Wong type contraction in a partial $S_{b}$ metric space.
 \end{example}
\begin{theorem}
  Let $(W, \wp )$ be a  complete  partial  $S_{b}$-metric space with coefficient $t \geq 1$ and   $S$ be an interpolative Boyd-Wong type contraction. Then $S$ has a unique fixed point $a$ (say) in $W$ and $\wp(a,a,a) = 0.$ 
\end{theorem}
\begin{proof}
    Let $a_{0} \in W.$ Define $S^{k}(a_{0})=a_{k}$ for each $k \in \mathbb{N}$. If $a_{k}=a_{k+1}$ for some $k \in \mathbb{N}$, then $a_{k}$ becomes fixed point of $S.$ Thus, without loss of generality, suppose $a_{k}\neq a_{k+1}$ for all $k \geq 0.$ For $a=b=a_{k}$ and $c=a_{k+1}$ in $(3.1)$, we get
    \begin{align}
      \wp (a_{k}, a_{k}, a_{k+1})= \wp (S(a_{k-1}), S(a_{k-1}), S(a_{k})) &\leq \tau ([\wp(a_{k-1},a_{k-1},a_{k})]^{p}[\wp(a_{k-1},a_{k-1},S(a_{k-1}))]^{q}[\wp(a_{k-1},a_{k-1},S(a_{k-1}))]^{r}\notag\\
[\wp(a_{k},a_{k},S(a_{k}))]^{s}&[\frac{1}{2t} (\wp(S(a_{k-1}),S(a_{k-1}), a_{k-1})+\wp(S(a_{k-1}),S(a_{k-1}), a_{k})]^{1-p-q-r-s})\notag\\
&\leq \tau ([\wp(a_{k-1},a_{k-1},a_{k})]^{p}[\wp(a_{k-1},a_{k-1},a_{k})]^{q}[\wp(a_{k-1},a_{k-1},a_{k})]^{r}\notag \\
&[\wp(a_{k},a_{k},a_{k+1})]^{s}[\frac{1}{2t} (\wp(a_{k},a_{k}, a_{k-1})+\wp(a_{k},a_{k}, a_{k})]^{1-p-q-r-s}).\end{align} Using $\wp(a_{k},a_{k},a_{k})\leq \wp(a_{k},a_{k},a_{k+1})$, $\wp(a_{k},a_{k},a_{k+1}) = \wp(a_{k+1},a_{k+1},a_{k})$ and the property of $\mathcal{J}$, we have
\begin{align}
 \wp (a_{k}, a_{k}, a_{k+1})< [\wp(a_{k-1},a_{k-1},a_{k})]^{p}[\wp(a_{k-1},a_{k-1},a_{k}&)]^{q}[\wp(a_{k-1},a_{k-1},a_{k})]^{r}
[\wp(a_{k},a_{k},a_{k+1})]^{s}\notag \\&~\hspace{1.5cm}[\frac{1}{2t} (\wp(a_{k},a_{k}, a_{k-1})+\wp(a_{k},a_{k}, a_{k})]^{1-p-q-r-s}\\
\leq  [\wp(a_{k-1},a_{k-1},a_{k})]^{p}[\wp(a_{k-1},a_{k-1},a_{k})&]^{q}[\wp(a_{k-1},a_{k-1},a_{k})]^{r}[\wp(a_{k},a_{k},a_{k+1})]^{s}\notag\\
&[\frac{1}{2} (\wp(a_{k-1},a_{k-1}, a_{k})+\wp(a_{k},a_{k}, a_{k+1})]^{1-p-q-r-s},
     \end{align}
   Suppose that  $\wp(a_{k},a_{k}, a_{k+1}) >\wp(a_{k-1},a_{k-1},a_{k})$ for some $k\in \mathbb{N}$ then $\frac{1}{2} (\wp(a_{k-1},a_{k-1}, a_{k})+\wp(a_{k},a_{k}, a_{k+1})) \leq \wp(a_{k},a_{k}, a_{k+1})$. Then by inequality $(3.4)$, we have
   $ \wp (a_{k}, a_{k}, a_{k+1}) <  \wp (a_{k}, a_{k}, a_{k+1}) $ which is absurd. Thus, we have $\wp(a_{k},a_{k}, a_{k+1}) \leq \wp(a_{k-1},a_{k-1},a_{k})$ for all $k \in \mathbb{N}.$ Hence, $\{ \wp(a_{k},a_{k}, a_{k+1})\}$ is a monotonically increasing sequence of positive real numbers. So, it converges to a limit v (say).
   \begin{equation}
   \mbox{Claim }1. ~v=0.\hspace{14.5cm} 
   \end{equation}
   Suppose $v \neq 0.$ Note that $\wp(a_{k},a_{k},a_{k})\leq \wp(a_{k},a_{k},a_{k+1})$. Hence, $\limsup_{k \rightarrow \infty }{\wp(a_{k},a_{k},a_{k})} \leq \limsup_{k \rightarrow \infty } {\wp(a_{k},a_{k},a_{k+1})}$. So, either $\limsup_{k \rightarrow \infty }{\wp(a_{k},a_{k},a_{k})} = v $ or  $\limsup_{k \rightarrow \infty }{\wp(a_{k},a_{k},a_{k})} =d < v$. Now we have two cases:\\
   Case $1$. If $\limsup_{k \rightarrow \infty }{\wp(a_{k},a_{k},a_{k})} = v $ holds then we know that $\tau$ is upper semi-continuous and from inequality $(3.2),$  
    $v = \lim_{k \rightarrow \infty } {\wp(a_{k},a_{k},a_{k+1})}   \leq \limsup_{k \rightarrow \infty }  \tau ([\wp(a_{k-1},a_{k-1},a_{k})]^{p}[\wp(a_{k-1},a_{k-1},a_{k})]^{q}[\wp(a_{k-1},a_{k-1},a_{k})]^{r}[\wp(a_{k},a_{k},a_{k+1})]^{s}$ \\
$[\frac{1}{2t} (\wp(a_{k},a_{k}, a_{k-1})+\wp(a_{k},a_{k}, a_{k})]^{1-p-q-r-s}) \leq \tau (v) < v $ which is absurd. So $v=0$.\\
Case $2$. If  $\limsup_{k \rightarrow \infty }{\wp(a_{k},a_{k},a_{k})} =d < v$ holds then from inequality $(3.3)$ we get 
 $v = \lim_{k \rightarrow \infty } {\wp (a_{k}, a_{k}, a_{k+1})}\leq  \limsup_{k \rightarrow \infty } ([\wp(a_{k-1},a_{k-1},a_{k})]^{p}[\wp(a_{k-1},a_{k-1},a_{k})]^{q}[\wp(a_{k-1},a_{k-1},a_{k})]^{r}
[\wp(a_{k},a_{k},a_{k+1})]^{s}[\frac{1}{2t} (\wp(a_{k},a_{k}, a_{k-1})+\wp(a_{k},a_{k},$\\$ a_{k})]^{1-p-q-r-s}.)\leq v^{p}v^{q}v^{r}v^{s}[\frac{1}{2t}(v+d)]^{1-p-q-r-s}< v.$ which is absurd. Hence, $v=0.$\\
Claim $2$. 
$\{a_{k}\}$ is a cauchy sequence.\\
Suppose 
$\{a_{k}\}$ is not cauchy sequence. So there is $ \delta > 0$ with subsequences of natural numbers
$\{k_{n}\}, 
\{l_{n}\}$ such that \begin{equation}
k_{n}> l_{n} \geq n 
\mbox{ with }\wp(a_{k_{n}},a_{k_{n}}, a_{l_{n}}) \geq \delta \mbox{ and } \wp(a_{l_{n}},a_{l_{n}}, a_{k_{n-1}})< \delta \hspace{7.15cm}
\end{equation}
For $a=a_{l_{n}}, b=a_{l_{n}}, c=a_{k_{n}}$ in inequality $(3.1)$, using the property of $\tau$ and inequalities $(3.5)$ and $(3.6)$,  we have \begin{align}
   \wp (a_{l_{n}},a_{l_{n}},a_{k_{n}}) = &\wp(S(a_{l_{n}-1}), S(a_{l_{n}-1}), S(a_{k_{n}-1})) \leq \tau ([\wp(a_{l_{n}-1},a_{l_{n}-1},a_{k_{n}-1})]^{p}[\wp(a_{l_{n}-1},a_{l_{n}-1},S(a_{l_{n}-1})]^{q}\notag\\&[\wp(a_{l_{n}-1},a_{l_{n}-1},S(a_{l_{n}-1}))]^{r}[\wp(a_{k_{n}-1},a_{k_{n}-1},S(a_{k_{n}-1}))]^{s}\notag\\&[\frac{1}{2t} (\wp(S(a_{l_{n}-1}),S(a_{l_{n}-1}), a_{l_{n}-1})+\wp(S(a_{l_{n}-1}),S(a_{l_{n}-1}),a_{k_{n}-1} )]^{1-p-q-r-s})\notag \\\mbox{Also, }\delta \leq \wp (a_{l_{n}},a_{l_{n}},&a_{k_{n}}) < [\wp(a_{l_{n}-1},a_{l_{n}-1},a_{k_{n}-1})]^{p}[\wp(a_{l_{n}-1},a_{l_{n}-1},a_{l_{n}}]^{q}[\wp(a_{l_{n}-1},a_{l_{n}-1},a_{l_{n}})]^{r}[\wp(a_{k_{n}-1},a_{k_{n}-1},a_{k_{n}}]^{s}\notag&\hspace{4cm}[\frac{1}{2t} (\wp(a_{l_{n}},a_{l_{n}}, a_{l_{n}-1})+\wp(a_{l_{n}},a_{l_{n}},a_{k_{n}-1} )]^{1-p-q-r-s}\notag\\
   \delta \leq \wp (a_{l_{n}},a_{l_{n}},a_{k_{n}}&) \leq [t(\wp(a_{l_{n}-1},a_{l_{n}-1},a_{l_{n}}) +\wp(a_{l_{n}-1},a_{l_{n}-1},a_{l_{n}}) +\wp(a_{k_{n}-1},a_{k_{n}-1},a_{l_{n}}))-\wp(a_{l_{n}},a_{l_{n}},a_{l_{n}})]^{p}\notag\\&[\wp(a_{l_{n}-1},a_{l_{n}-1},a_{l_{n}}]^{q}[\wp(a_{l_{n}-1},a_{l_{n}-1},a_{l_{n}})]^{r}[\wp(a_{k_{n}-1},a_{k_{n}-1},a_{k_{n}}]^{s}[\frac{1}{2t} (\wp(a_{l_{n}},a_{l_{n}}, a_{l_{n}-1})+ \delta )]^{1-p-q-r-s }\notag\\
    \delta \leq \wp (a_{l_{n}},a_{l_{n}},a_{k_{n}}&) \leq [t(\wp(a_{l_{n}-1},a_{l_{n}-1},a_{l_{n}}) +\wp(a_{l_{n}-1},a_{l_{n}-1},a_{l_{n}})+ \delta) -\wp(a_{l_{n}},a_{l_{n}},a_{l_{n}})]^{p}[\wp(a_{l_{n}-1},a_{l_{n}-1},a_{l_{n}}]^{q}\notag\\&~\hspace{1cm}[\wp(a_{l_{n}-1},a_{l_{n}-1},a_{l_{n}})]^{r}[\wp(a_{k_{n}-1},a_{k_{n}-1},a_{k_{n}}]^{s}[\frac{1}{2t} (\wp(a_{l_{n}},a_{l_{n}}, a_{l_{n}-1})+ \delta )]^{1-p-q-r-s }.\notag
   \end{align}
   Taking $\limsup_{n\rightarrow \infty} $ and using inequality $(3.5)$, we get $\delta=0$ which is absurd as $\delta > 0.$ Thus, $\{a_{k}\}$ is a cauchy sequence.  \begin{equation} \mbox{Since } W \mbox{ is complete, therefore, there is } a  \mbox{ such that } \wp(a,a,a)=\lim_{k\rightarrow \infty}\wp(a_{k},a_{k},a)=0.\hspace{3.8cm}\end{equation}
   For $a=a_{k},b=a_{k},c=a$ in $(3.1)$ and the property of $\mathcal{J}$ we get \\
$\wp(a_{k+1},a_{k+1} ,a)  = \wp(S(a_{k}), S(a_{k}), S(a))  \leq \mathcal{J}([\wp(a_{k},a_{k},a)]^{p}[\wp(a_{k},a_{k},S(a_{k})]^{q}[\wp(a_{k},a_{k},S(a_{k})]^{r}[\wp(a,a,S(a))]^{s}$\\$~\hspace{7cm}[\frac{1}{2t} (\wp(S(a_{k}),S(a_{k}), a_{k})+\wp(S(a_{k}),S(a_{k}), a)]^{1-p-q-r-s})$\\
$\wp(a_{k+1},a_{k+1} ,a)  \leq  [\wp(a_{k},a_{k},a)]^{p}[\wp(a_{k},a_{k},a_{k+1}]^{q}[\wp(a_{k},a_{k},a_{k+1}]^{r}[\wp(a,a,S(a))]^{s}$\\$~\hspace{7cm}[\frac{1}{2t} (\wp(a_{k+1},a_{k+1}, a_{k+1})+\wp(a_{k+1},a_{k+1}, a)]^{1-p-q-r-s}$\\
Taking $\limsup_{k\rightarrow \infty} $, using upper semicontinuity of $\tau$ and using inequality $(3.5)$ and $(3.7)$, we get \\
$\limsup_{k\rightarrow \infty} \wp(a_{k+1},a_{k+1} ,S (a))  \leq \tau(0) =0 $. Hence, $\limsup_{k\rightarrow \infty} \wp(a_{k+1},a_{k+1} ,S( a))=0.$ As $\wp(S (a), S (a) , S (a)) \leq \wp(W (a), S (a) ,a_{k+1})  =\wp(a_{k+1},a_{k+1} ,S (a)).$ So, $\wp(S(a), S(a) , S(a)) \leq \limsup_{k\rightarrow \infty} \wp(a_{k+1},a_{k+1} ,S (a)) =0$. Thus, $\wp(S (a), S (a) , S (a))=0.$ 
Also, $\wp(a,a,S (a)) \leq t[\wp (a,a,a_{k}) + \wp (a,a,a_{k}) + \wp(S (a) ,S (a) ,a_{k})]- \wp(a_{k},a_{k}, a_{k}).$ \\
Taking $\limsup_{k\rightarrow \infty}$ and using inequalities $(3.5)$, we get $\wp(a,a,S (a) )=0$. \\
Thus $\wp(a,a,S (a) )=\wp(a,a, a )=\wp(S (a), S (a),S (a) )=0$. Hence, $a=W (a).$\\
Claim $3.~a$ is the unique fixed point of $S$.\\
Suppose $v$ is another fixed point of $S$.\\
Now, $\wp(a, a, v)=\wp(S(a), S(a), S(v))  \leq \tau ([\wp(a,a,v)]^{p}[\wp(a,a,S(a)]^{q}[\wp(a,a, S(a))]^{r}[\wp(v,v,S(v))]^{s}$\\$~\hspace{7cm}[\frac{1}{2t} (\wp(S(a),S(a), a)+\wp(S(a),S(a), v)]^{1-p-q-r-s})$\\
$~\hspace{5.7cm} \leq \tau ([\wp(a,a,v)]^{p}[\wp(a,a,a)]^{q}[\wp(a,a, a)]^{r}[\wp(v,v,v)]^{s}$\\$~\hspace{7cm}[\frac{1}{2t} (\wp(a,a, a)+\wp(a,a, v)]^{1-p-q-r-s})$
\\As $\wp(a,a,a)=0, ~ \wp(a, a, v) \leq \tau (0)=0.$ So, $\wp(a,a,v) =0$.
Since $\wp(v,v,v) \leq \wp(v,v,a) =\wp(a,a,v),~\wp(v,v,v)=0$. Thus 
$\wp(v,v,v) =\wp(a,a,a)=\wp(a,a,v) =0, \mbox{ which implies } a=v.$
\end{proof}
\begin{example}
Let $W= \{0,3\} \cup [4,\infty)$.Then  we define $\wp: W \times W \times W \rightarrow [0,\infty)$ as follows:\begin{center}
$ \wp(a,b,c)=\begin{cases}
    a^{5}, & \text{if $a=b=c$},\\
2(a^{5}+c^{5}), & \text{if $a=b\neq c$},\\
a^{5}+b^{5}+c^{5}, & \text{otherwise}.
  \end{cases} $\end{center}
  Clearly, $(W, \wp)$ is a complete partial  $S_{b}$-metric space with coefficient $t=1$.\begin{center}$\tau(a)=\begin{cases}
    \frac{9a}{10}, & \text{if $a\in [0,1]$},\\
    \frac{a}{2}, & \text{if $a>1$}.
  \end{cases}$\end{center} Clearly, $\tau  \in \Re$.
  Also, we define $S : W \rightarrow W$ as follows: \begin{center} $S (a)=\begin{cases}
    0, & \text{if $a\in \{0,3\}$},\\
    3, & \text{otherwise}.
  \end{cases}$\end{center}
  Clearly, $S$ is an interpolative Boyd-Wong type contraction by Example $3.2$.
  Note that $0$ is the unique fixed point of $S$.
\end{example}

\begin{definition}
  Let $(W, \wp )$ be a partial  $S_{b}$-metric space with coefficient $t \geq 1$. Then the operator $S$ from $W$ to $W$ is called an interpolative Matkowski type contraction in    partial  $S_{b}$-metric space, if there are $p,q,r,s \in (0,1)$ such that $p+q+r+s < 1$ and a monotonically increasing function $\mathcal{J} \in \mathcal{G}$ with 
  \begin{align}
  \wp(S(a), S(b), S(c))  \leq \mathcal{J} ([\wp(a,b,c)]^{p}[\wp(a,a,S(a))]^{q}[\wp(b,b,& S(b))]^{r}[\wp(c,c,S(c))]^{s}\notag\\&[\frac{1}{2t} (\wp(S(a),S(a), b)+\wp(S(b),S(b), c)]^{1-p-q-r-s})
  \end{align}
  for all $a,b,c \in W \backslash Fix(S).$
\end{definition}
\begin{example}
Let $W= \{0,3\} \cup [4,\infty)$.Then  we define $\wp: W \times W \times W \rightarrow [0,\infty)$ as follows : \begin{center}
$ \wp(a,b,c)=\begin{cases}
    a^{5}, & \text{if $a=b=c$},\\
2(a^{5}+c^{5}), & \text{if $a=b\neq c$},\\
a^{5}+b^{5}+c^{5}, & \text{otherwise}.
  \end{cases} $\end{center}
  Clearly, $(W, \wp)$ is a  partial  $S_{b}$-metric space with coefficient $t=1$. \\Now, define $\mathcal{J}: [0,\infty) \rightarrow [0,\infty)
  $ as $\mathcal{J}(a) = \frac{a}{2}.$ Clearly, $\mathcal{J} \in \mathcal{G}
  .$  Also, define $S : W \rightarrow W$ as follows:\begin{center}  $S (a)=\begin{cases}
    0, & \text{if $a\in \{0,3\}$},\\
    3, & \text{otherwise}.
  \end{cases}$\end{center}
  Clearly, $S$ is a interpolative Matkowski type contraction proved similarly as in  Example $3.2$ with $\tau$ replaced by $\mathcal{J}$.\end{example}
  \begin{theorem}
  Let $(W, \wp )$ be a  complete  partial  $S_{b}$-metric space with coefficient $t \geq 1$ and   $S$ be an interpolative Matkowski type contraction. Then $S$ has a unique fixed point $a$ (say) in $W$ and $\wp(a,a,a) = 0.$ 
\end{theorem}
\begin{proof}
 Let $a_{0} \in W.$ Define $S^{k}(a_{0})=a_{k}$ for each $k \in \mathbb{N}$. If $a_{k}=a_{k+1}$ for some $k \in \mathbb{N}$, then $a_{k}$ becomes fixed point of $S.$ Thus, without loss of generality, suppose $a_{k}\neq a_{k+1}$ for all $k \geq 0.$ For $a=b=a_{k}$ and $c=a_{k+1}$ in $(3.8)$, we get
    \begin{align}
      \wp (a_{k}, a_{k}, a_{k+1})= \wp (S(a_{k-1}), S(a_{k-1}), S(a_{k})) &\leq \mathcal{J} ([\wp(a_{k-1},a_{k-1},a_{k})]^{p}[\wp(a_{k-1},a_{k-1},S(a_{k-1}))]^{q}[\wp(a_{k-1},a_{k-1},S(a_{k-1}))]^{r}\notag\\
[\wp(a_{k},a_{k},S(a_{k}))]^{s}&[\frac{1}{2t} (\wp(S(a_{k-1}),S(a_{k-1}), a_{k-1})+\wp(S(a_{k-1}),S(a_{k-1}), a_{k})]^{1-p-q-r-s})\notag\\
&\leq \mathcal{J} ([\wp(a_{k-1},a_{k-1},a_{k})]^{p}[\wp(a_{k-1},a_{k-1},a_{k})]^{q}[\wp(a_{k-1},a_{k-1},a_{k})]^{r}\notag \\
&[\wp(a_{k},a_{k},a_{k+1})]^{s}[\frac{1}{2t} (\wp(a_{k},a_{k}, a_{k-1})+\wp(a_{k},a_{k}, a_{k})]^{1-p-q-r-s}).\end{align} By $\wp(a_{k},a_{k},a_{k})\leq \wp(a_{k},a_{k},a_{k+1})$, $\wp(a_{k},a_{k},a_{k+1}) = \wp(a_{k+1},a_{k+1},a_{k})$ and the property of $\mathcal{J}$ in $(3.9)$ we get
\begin{align}
 \wp (a_{k}, a_{k}, a_{k+1})\leq \mathcal{J} ( [\wp(a_{k-1},a_{k-1},a_{k})]^{p}[\wp(a_{k-1},a_{k-1},&a_{k})]^{q}[\wp(a_{k-1},a_{k-1},a_{k})]^{r}
[\wp(a_{k},a_{k},a_{k+1})]^{s}\notag \\&[\frac{1}{2t} (\wp(a_{k-1},a_{k-1}, a_{k})+\wp(a_{k+1},a_{k+1}, a_{k})]^{1-p-q-r-s}) \end{align}
Suppose that  $\wp(a_{k},a_{k}, a_{k+1}) >\wp(a_{k-1},a_{k-1},a_{k})$ for some $k\in \mathbb{N}$ then $\frac{1}{2} (\wp(a_{k-1},a_{k-1}, a_{k})+\wp(a_{k},a_{k}, a_{k+1})) \leq \wp(a_{k},a_{k}, a_{k+1})$. Then by inequality $(3.10)$ and using $\mathcal{J}(a)<a $, we have
   $ \wp (a_{k}, a_{k}, a_{k+1}) <  \wp (a_{k}, a_{k}, a_{k+1}) $ which is absurd. Thus, we have $\wp(a_{k},a_{k}, a_{k+1}) \leq \wp(a_{k-1},a_{k-1},a_{k})$ for all $k \in \mathbb{N}.$ Hence, $\{ \wp(a_{k},a_{k}, a_{k+1})\}$ is a monotonically increasing sequence of positive real numbers. So, it converges to a limit v (say). Note that by inequality $(3.10)$ and using $\wp(a_{k},a_{k}, a_{k+1}) \leq \wp(a_{k-1},a_{k-1},a_{k})$ we get $\wp(a_{k},a_{k}, a_{k+1}) \leq \mathcal{J}(\wp(a_{k-1},a_{k-1},a_{k}))$. On recurring the procedure, we get $\wp(a_{k},a_{k}, a_{k+1}) \leq \mathcal{J}^{k}(\wp(a_{0},a_{0},a_{1}))$. Taking $\lim_{k\rightarrow \infty}$  and using the property of $\mathcal{J}$, we get $v=\lim_{k\rightarrow \infty}\wp(a_{k},a_{k}, a_{k+1}) \leq \lim_{k\rightarrow \infty}\mathcal{J}^{k}(\wp(a_{0},a_{0},a_{1}))=0$.\\
The remaining proof is similar to that of the  Theorem $3.3$.
\end{proof}
\begin{example}
Let $W= \{0,3\} \cup [4,\infty)$. Then,  we define $\wp: W \times W \times W \rightarrow [0,\infty)$ as \begin{center}
$ \wp(a,b,c)=\begin{cases}
    a^{5}, & \text{if $a=b=c$},\\
2(a^{5}+c^{5}), & \text{if $a=b\neq c$},\\
a^{5}+b^{5}+c^{5}, & \text{otherwise}.
  \end{cases} $\end{center}
  Clearly, $(W, \wp)$ is a complete partial  $S_{b}$-metric space with coefficient $t=1$. Now, define $\mathcal{J}: [0,\infty) \rightarrow [0,\infty)
  $ as $\mathcal{J}(a) = \frac{a}{2}.$ Clearly, $\mathcal{J} \in \mathcal{G}
  .$ Also, define $S : W \rightarrow W $ as   \begin{center}$S (a)=\begin{cases}
    0, & \text{if $a\in \{0,3\}$},\\
    3, & \text{otherwise}.
  \end{cases}$\end{center}
  Clearly, $S$ is an interpolative Matkowski type contraction proved similarly as in Example $3.2$ with $\tau$ replaced by $\mathcal{J}$. Note that $0$ is the unique fixed point of $W$.\end{example}
\section*{Conclusion}
We have introduced the topological structure on  partial  $S_{b}-$metric space. Further,
we have noticed that a 
  partial  $S_{b}-$metric topological space is a $T_{0}$ space but need not be connected, $T_{1}$
 space, $T_{2}$ space and compact space. Now we can discuss the topological properties like convergence, continuity  and many more such. Also, we have given the definition of interpolative Boyd-Wong type contraction and interpolative Matkowski type contaction in the setting of partial  $S_{b}-$metric space and proved some related fixed point theorems on them.

\end{document}